\documentclass[10pt]{amsart}
\usepackage{amsthm}
 
\usepackage{}%
\usepackage{latexsym}
\usepackage{amscd}
\usepackage{amssymb}
\usepackage{graphicx}
\usepackage{epstopdf}
\usepackage{amsmath}
\usepackage{textcomp}
\usepackage{amsthm}
\usepackage{bm}
\usepackage{xcolor,soul}
\usepackage{color}
\definecolor{darkblack}{rgb}{0,0,0} 

\usepackage{microtype}
\usepackage[normalem]{ulem}
\usepackage{cancel}
\usepackage{tikz}
\usetikzlibrary{decorations.markings}
\tikzstyle{vertex}=[circle, draw, inner sep=0pt, minimum size=6pt]

\usepackage{caption}
\usepackage{subcaption}
\usepackage{amssymb, amscd, amsmath,amsthm,amsfonts,epsfig, enumerate}
\usepackage{epsfig}
\usepackage{fullpage}
\usepackage{listings}
\usepackage{float}
\usepackage{epstopdf}
\usepackage{color}
\usepackage{array}
\usepackage{booktabs}
\usepackage{xcolor}
\usepackage{geometry}
\definecolor{whiskerblue}{RGB}{0,114,189}
\definecolor{pathblue}{RGB}{0,0,255}
\definecolor{vertexred}{RGB}{255,0,0}
\def\B'c{{\mathcal{B'}}}
\def\U'c{{\mathcal{U'}}}
\def\opn#1#2{\def#1{\operatorname{#2}}} 
\opn\chara{char}
\opn\length{\ell}
\opn\projdim{proj\,dim}
\opn\injdim{inj\,dim}
\opn\ini{in}
\opn\rank{rank}
\opn\depth{depth}
\opn\sdepth{sdepth}
\opn\indmat{indmat}
\opn\cochord{cochord}
\opn\pdim{pdim}
\opn\height{ht}
\opn\embdim{emb\,dim}
\opn\codim{codim}

\opn\Tr{Tr}
\opn\bigrank{big\,rank}
\opn\superheight{superheight}\opn\lcm{lcm}
\opn\trdeg{tr\,deg}%
\opn\reg{reg}
\opn\lreg{lreg}
\opn\set{set}
\opn\supp{Supp}
\opn\shad{Shad}
\opn\div{div}
\opn\Div{Div}
\opn\cl{cl}
\opn\Cl{Cl}
\opn\Spec{Spec}
\opn\Supp{Supp}
\opn\supp{supp}
\opn\Sing{Sing}
\opn\Ass{Ass}
\opn\Min{Min}
\opn\size{size}
\opn\bigsize{bigsize}
\opn\lex{lex}
\opn\Ann{Ann}
\opn\Rad{Rad}
\opn\Soc{Soc}
\opn\Ker{Ker}
\opn\Coker{Coker}
\opn\Im{Im}
\opn\Hom{Hom}
\opn\Tor{Tor}
\opn\Ext{Ext}
\opn\End{End}
\opn\Aut{Aut}
\opn\id{id}

\opn\nat{nat}
\opn\GL{GL}
\opn\SL{SL}
\opn\mod{mod}
\opn\ord{ord}
\opn\aff{aff}
\opn\con{conv}
\opn\relint{relint}
\opn\st{st}
\opn\lk{lk}
\opn\cn{cn}
\opn\core{core}
\opn\vol{vol}
\opn\gr{gr}

\def\pot#1#2{#1[\kern-0.28ex[#2]\kern-0.28ex]}

\opn\dirlim{\underrightarrow{\lim}}
\opn\invlim{\underleftarrow{\lim}}
\let\tensor=\otimes

\def\pnt{{\raise0.5mm\hbox{\large\bf.}}}

\def\Implies{\ifmmode\Longrightarrow \else
	\unskip${}\Longrightarrow{}$\ignorespaces\fi}
\def\implies{\ifmmode\Rightarrow \else
	\unskip${}\Rightarrow{}$\ignorespaces\fi}
\def\iff{\ifmmode\Longleftrightarrow \else
	\unskip${}\Longleftrightarrow{}$\ignorespaces\fi}

\let\:=\colon
\newtheorem{Theorem}{Theorem}[section]
\newtheorem{Lemma}[Theorem]{Lemma}
\newtheorem{Corollary}[Theorem]{Corollary}

\newtheorem{Remark}[Theorem]{Remark}

\newtheorem{Example}[Theorem]{Example}

\let\epsilon=\varepsilon
\let\phi=\varphi
\let\kappa=\varkappa
\numberwithin{equation}{section}
\title{Castelnuovo-Mumford Regularity and Combinatorial Invariants of Trees}

\author[Ahtsham ul Haq$^1$]{Ahtsham ul Haq$^1$}

\author[Muhammad Usman Rashid$^{1,2}$]{Muhammad Usman Rashid$^{1,2} $}

\author[Muhammad Ishaq$^{1,*}$]{Muhammad Ishaq$^{1,*}$}

\begin{document}
	\maketitle
 \begin{center}
     $^1$School of Natural Sciences, National University of Sciences and Technology, Islamabad, Pakistan.
     $^2$Department of Sciences and Humanities, National University of Computer and Emerging Sciences, Islamabad, Pakistan.
     \linebreak
     $^*$ Corresponding author email: \email{ishaq\_maths@yahoo.com}\par
     Contributing authors email: ahtsham2192@gmail.com, usmanrashiid@gmail.com
     \end{center}
    \begin{abstract}
    This work establishes combinatorial bounds on the Castelnuovo-Mumford regularity of edge ideals for trees and their multi-whiskered variants. For a tree \( T \), we give bounds for the Castelnuovo-Mumford regularity of \( I(T) \) in terms of the order, diameter, and number of pendant vertices. Furthermore, we present an upper bound for multi-whiskered trees \( T_{\mathbf{a}} \), demonstrating that the Castelnuovo-Mumford regularity of \( I(T_{\mathbf{a}}) \) is bounded by the same invariants of the underlying tree \( T \). A principal consequence of this work is the derivation of corresponding inequalities for two key combinatorial invariants of \( T \), namely the induced matching number \( \operatorname{im}(T) \) and the independence number \( \alpha(T) \).
	\end{abstract}
\section{Introduction}\label{sec:introduction}
Let $S = K[x_1, \dots, x_n]$ be the polynomial ring in $n$ variables over a field $K$ with the standard grading, and let $M$ be a finitely generated graded $S$-module. Consider $M$ admits a minimal graded free resolution
\[
0 \longrightarrow \bigoplus_{j \in \mathbb{Z}} S(-j)^{\beta_{p,j}(M)} \longrightarrow \cdots \longrightarrow \bigoplus_{j \in \mathbb{Z}} S(-j)^{\beta_{1,j}(M)} \longrightarrow \bigoplus_{j \in \mathbb{Z}} S(-j)^{\beta_{0,j}(M)} \longrightarrow M \longrightarrow 0,
\]
where $p = \operatorname{pdim}(M)$ denotes the projective dimension of $M$. The integers $\beta_{i,j}(M)$, called the \emph{graded Betti numbers} of $M$, count the number of minimal generators of degree $j$ in the $i$-th syzygy module. The \emph{Castelnuovo-Mumford regularity} of $M$ is defined as
\[
\operatorname{reg}(M) = \max\{j - i : \beta_{i,j}(M) \neq 0\}.
\]

Let $G$ be a simple graph with vertex set $V(G) = \{x_1, \dots, x_n\}$ and edge set $E(G)$. The \emph{edge ideal} of $G$ is the squarefree monomial ideal $I(G) = (x_ix_j : \{x_i, x_j\} \in E(G)) \subset S$. The study of Castelnuovo-Mumford regularity for edge ideals is a central topic in combinatorial commutative algebra \cite{dao, hibikrull, fakhari, circulent}, with particular interest in understanding how graph-theoretic properties influence algebraic invariants. For bipartite graphs, significant progress includes Kummini's determination of Castelnuovo-Mumford regularity for Cohen--Macaulay bipartite graphs \cite{kumini2009}, extended by Van Tuyl to sequentially Cohen--Macaulay bipartite graphs \cite{VTSequentially}. Further advances include results for very well-covered graphs \cite{Mahmoudi}, bounds for vertex-decomposable and shellable graphs \cite{Moradi}, and studies of various other graph classes \cite{JCT1, JCT2,woodroof}. 


A \emph{path graph} $P_n$ is a graph with vertex set $\{v_1, \dots, v_n\}$ and edge set $\{v_iv_{i+1} : 1 \leq i \leq n-1\}$. The \emph{diameter} of a graph $G$, denoted $d(G)$, is the maximum distance between any two vertices. A \emph{pendant vertex} (or \emph{leaf}) is a vertex of degree one; we denote by $p(G)$ the number of pendant vertices in $G$. A \textit{tree} is a connected acyclic graph. These graphs have been extensively studied for both their graph-theoretic properties and the algebraic properties of their edge ideals \cite{kumini2009, Ahtsham, VTSequentially, woodroof}.

For the sake of brevity, we will use $d$ for $d(G)$, and $p$ for $p(G)$. In this paper, we establish combinatorial bounds for the Castelnuovo-Mumford regularity of $S/I(T)$, where $S=K[V(T)]$ and $T$ is a tree. For a tree $T$ of order $n \geq 2$ with diameter $d$ and $p$ pendant vertices, Theorems~\ref{Theorem.LB-reg,tree} and~\ref{Theorem.UB-reg,tree} prove that
\[
\left\lfloor \frac{n - p + d + 5}{6} \right\rfloor \leq \operatorname{reg}(S/I(T)) \leq \min\left\{n - p, \left\lfloor \frac{2n - p}{3} \right\rfloor\right\}.
\]

A \textit{whisker} in a graph $G$ is an edge formed by adding a new vertex and connecting it to an existing vertex of $G$. The \emph{whiskered graph} of $G$, obtained by adding a whisker at each vertex, was introduced by Villarreal \cite{vil1990}. Recently, this concept was generalized to \emph{multi-whiskered graphs} by Muta et al. \cite{Muta}, where multiple whiskers can be attached to each vertex. For a graph $G$ with vertex set $V(G) = \{x_1, \dots, x_n\}$, the multi-whiskered graph is denoted by $G_{\mathbf{a}}$, where $\mathbf{a} = (a_1, \dots, a_n) \in \mathbb{Z}_{+}^n$, and $a_i$ represents the number of pendant vertices attached to $x_i \in V(G)$. When $a_i = 1$ for all $i=1,\dots,n$, we obtain the whiskered graph $G_{\mathbf{1}}$.

For a tree $T$ of order $n \geq 2$ with diameter $d$ and $p$ pendant vertices, Theorem~\ref{combine-UB-1Whisker-p-d} establishes the following bound for the Castelnuovo-Mumford regularity of $S/I(T_{\mathbf{a}})$, where $S = K[V(T_{\mathbf{a}})]$ and $T_{\mathbf{a}}$ is a multi-whiskered tree:
\[
\operatorname{reg}(S/I(T_{\mathbf{a}})) \leq \min\left\{\left\lceil\frac{2n-d-1}{2}\right\rceil, \left\lfloor\frac{2n+p-2}{3}\right\rfloor\right\}.
\]

A \emph{matching} in a graph is a set of pairwise non-adjacent edges. A matching $M$ is \emph{induced} if no two edges in $M$ are connected by an edge in the graph. The \emph{induced matching number} $\operatorname{im}(G)$ is the maximum size of an induced matching in $G$. Katzman \cite{kat} proved that $\operatorname{reg}(S/I(G)) \geq \operatorname{im}(G)$. For chordal graphs, Ha et al. \cite{hanvan} showed that equality holds. Since trees are chordal, we have $\operatorname{reg}(S/I(T)) = \operatorname{im}(T)$, and consequently our bounds immediately yields
\[ 
\left\lfloor \frac{n-p+d+5}{6} \right\rfloor \leq \operatorname{im}(T) \leq \min\left\{n-p, \left\lfloor \frac{2n-p}{3} \right\rfloor\right\}.
\] 

An \emph{independent set} in $G$ is a set of pairwise non-adjacent vertices, and the \emph{independence number} $\alpha(G)$ is the maximum size of an independent set. For multi-whiskered trees, there exists a well-known correspondence between the Castelnuovo-Mumford regularity of $I(T_{\mathbf{a}})$ and the independence number of $T$. Specifically, for any tree $T$ and multi-whiskering $T_{\mathbf{a}}$, we have $\operatorname{reg}(S/I(T_{\mathbf{a}})) = \alpha(T)$. Applying this result to our bound for multi-whiskered trees yields
\[
\alpha(T) \leq \min\left\{\left\lceil\frac{2n-d-1}{2}\right\rceil,\left\lfloor\frac{2n+p-2}{3}\right\rfloor\right\}.
\]

The paper is organized as follows. Section 2 provides necessary background on graph theory and homological algebra. In Section 3, we prove Theorems~\ref{Theorem.LB-reg,tree} and~\ref{Theorem.UB-reg,tree}, establishing combinatorial bounds for $\operatorname{reg}(S/I(T))$. As an immediate consequence, we derive corresponding bounds for the induced matching number of trees (Corollary~\ref{cor1}). Section 4 extends these results to multi-whiskered trees (Theorem~\ref{combine-UB-1Whisker-p-d}) and establishes an upper bound for the independence number of trees (Corollary~\ref{cor,WT}). The paper concludes with Section 5, where we discuss broader implications and suggest avenues for future research.
\section{Notations and Preliminaries}\label{Section 2}
We recall fundamental definitions from graph theory and key results from commutative algebra that will be used throughout this paper. Unless otherwise specified, all graphs are simple and finite.

\subsection{Graph-Theoretic Notions}
A \textit{$k$-star}, denoted $\mathcal{S}_k$, is a tree consisting of one central vertex and $k$ pendant vertices. Note that $\mathcal{S}_1$ is isomorphic to $P_2$, the path on two vertices. A \textit{bistar graph}, denoted $\mathcal{B}_{k_1,k_2}$, is obtained by joining the central vertices of two stars $\mathcal{S}_{k_1}$ and $\mathcal{S}_{k_2}$ with an edge. The resulting graph has $k_1 + k_2$ pendant vertices.
For any graph $G$, the \textit{neighborhood} of a vertex $x_i$ is defined as $N_G(x_i) = \{x_j \in V(G) : \{x_i, x_j\} \in E(G)\}$. The \textit{closed neighborhood} is $N_G[x_i] = N_G(x_i) \cup \{x_i\}$. The \textit{degree} of a vertex $x_i$, denoted $\deg_G(x_i)$, is the cardinality of $N_G(x_i)$. 
A vertex adjacent to a pendant vertex is called a \textit{support vertex}. If a support vertex is adjacent to two or more pendant vertices, it is called a \textit{strong support vertex}. We denote the set of support vertices by $Q(G)$ and the set of pendant vertices by $L(G)$. Note that $|L(G)| = p(G)$, where $p(G)$ denotes the number of pendant vertices as defined in the introduction.

\subsection{Algebraic Background }
We will frequently use the following fact without explicit reference: for any monomial ideal $I \subset S$, if $\hat{S} = S \otimes_K K[x_{n+1}]$, then $\reg(\hat{S}/I) = \reg(S/I)$ \cite[Lemma 3.6]{moreyvir}.

\begin{Lemma}[{\cite[Lemma 2.10]{dao}}]\label{comareg}
Let $I \subset S = K[x_1, \dots, x_n]$ be a monomial ideal and $x_i$ be a variable of $S$. Then:
\begin{itemize}
\item[\textbf{a.}] $\reg(S/(I, x_i)) \leq \reg(S/I)$.
\item[\textbf{b.}] $\reg(S/I) \leq \max\{\reg(S/(I:x_i)) + 1, \reg(S/(I, x_i))\}$.
\end{itemize}
\end{Lemma}

\begin{Lemma}[{\cite[Lemma 3.2]{HOA}}]\label{circulentt}
Let $1 \leq r < n$, $I \subset S_1 = K[x_1, \dots, x_r]$ and $J \subset S_2 = K[x_{r+1}, \dots, x_n]$ be nonzero homogeneous ideals in disjoint sets of variables. Then
\[
\reg(S/(I + J)) = \reg(S_1/I) + \reg(S_2/J).
\]
\end{Lemma}

\begin{Lemma}[{\cite[Corollary 6.9]{hanvan}}]\label{reg1}
If $G$ is a chordal graph, then $\operatorname{reg}(S/I(G)) = \operatorname{im}(G)$.
\end{Lemma}

\begin{Lemma}[{\cite[Corollary 4.2]{Mujahid}}]\label{independence}
If \( G \) is a graph, then \(\operatorname{indmat}(G_{\mathbf{a}}) = \alpha(G)\).
\end{Lemma}
\section{castelnuovo-mumford regularity of the edge ideals of trees}\label{Section 3}
This section derives sharp combinatorial bounds for the Castelnuovo-Mumford regularity of $S/I(T)$, where $T$ is a tree. We prove that $\operatorname{reg}(S/I(T))$ is determined by three fundamental graph parameters, namely the order $n$, diameter $d$, and number of pendant vertices $p$ of $T$. Our results reveal how the homological complexity of the edge ideal is governed by the combinatorial structure of the underlying tree, with proofs that combine commutative algebra and graph theory.
\begin{Theorem}\label{Theorem.LB-reg,tree}
Let $T$ be a tree of order \( n \geq 2 \), diameter $d$ and pendant vertices $p$. If $S=K[V(T)]$ and $I=I(T)$, then
\[\reg(S/I) \geq \left\lfloor \frac{n-p+d+5}{6} \right\rfloor.\]
\end{Theorem}

\begin{proof}
For \( 2 \leq n \leq 6 \), the result is immediate from Lemma \ref{reg1}. For \( n = 7 \), verification across all non-isomorphic trees confirms the inequality (see Table \ref{All7VertexTrees}). Now assume \( n \geq 8 \), and consider the following special cases:\\
\textbf{i.} If \( p = 2, \)  then \( T \cong P_n \), and \( d = n-1 \). By {\cite[Proposition 10]{woodroof}}, 
\[
\reg(S/I)= \left\lfloor \frac{n + 1}{3} \right\rfloor = \left\lfloor \frac{n - 2 + (n - 1) + 5}{6} \right\rfloor = \left\lfloor \frac{n-p+d+5}{6} \right\rfloor. 
\]
\textbf{ii.} If \( d = 2, \)  then \( T \cong \mathcal{S}_p\), and $p=n-1$. By Lemma \ref{reg1},
\[
\reg(S/I)=1=\left\lfloor \frac{n-(n-1)+2+5}{6} \right\rfloor =  \left\lfloor \frac{n-p+d+5}{6} \right\rfloor.
\]
\textbf{iii.} If \( d = 3, \)  then \( T \cong \mathcal{B}_p, \) and \( n-p=2. \) By Lemma \ref{reg1},
\[
\reg(S/I)= 1 = \left\lfloor \frac{2 + 3 + 5}{6} \right\rfloor = \left\lfloor \frac{n-p+d+5}{6} \right\rfloor. 
\]
Now assume $d \geq 4$ and $p \geq 3.$ Let $P_{d+1}$ be an induced path that realizes the diameter in $T$, and label the vertices of $P_{d+1}$ by $x_1,x_2,\ldots,x_{d+1}$ (where $x_i$ is adjacent to $x_{i+1}$ for all $1\leq i\leq d$). Let $L(T):=\{y_{1},\dots,y_{p-2}\}\cup \{x_1, x_{d+1}\}$  be the set of all pendant vertices, and $Q(T)$ be the set of all support vertices in $T.$ 
The proof is divided into two cases: 
\begin{itemize}
    \item [1.] Every support vertex in $T$ has degree $2$.
    \item [2.] $T$ admits a support vertex of degree $\geq 3$.
\end{itemize}
\textbf{Case 1.} 
Assume every support vertex has degree \(2\). Then \(|L(T)| = |Q(T)|\). Let \(z_i\) be the unique neighbor of \(y_i\) for \(i = 1, \dots, p-2\). Then the support vertices are \(Q(T) = \{z_1, \dots, z_{p-2}\} \cup \{x_2, x_d\}\).
Consider the subcase \(d = 4\). In such a tree, the total number of vertices is $n = 2p+1$, and \(\operatorname{im}(T) = p\). By Lemma~\ref{reg1}, \(\operatorname{reg}(S/I) = p\). Then,
\[
\operatorname{reg}(S/I) = p \geq \left\lfloor \frac{p + 10}{6} \right\rfloor = \left\lfloor \frac{2p + 1 - p + 4 + 5}{6} \right\rfloor = \left\lfloor \frac{n - p + d + 5}{6} \right\rfloor,
\]
which proves the inequality for \(d = 4\).

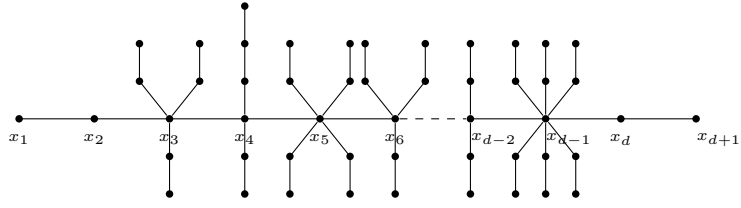
\begin{figure}[H]
\centering
\begin{tikzpicture}[scale=1,every node/.style={circle, fill, inner sep=1pt, minimum size=2pt}]
\draw
    (0,0) node[label=below:\tiny{$x_1$}] (x1) {} -- 
    (1,0) node[label=below:\tiny$x_2$] (x2) {} -- 
    (2,0) node[label=below:\tiny$x_3$] (x3) {} -- 
    (3,0) node[label=below:\tiny$x_4$] (x4) {} -- 
    (4,0) node[label=below:\tiny$x_5$] (x5) {} -- 
    (5,0) node[label=below:\tiny$x_6$] (x6) {} 
    (6,0) node[label=340:\tiny$x_{d-2}$] (x7) {} -- 
    (7,0) node[label=340:\tiny$x_{d-1}$] (x8) {} -- 
    (8,0) node[label=below:\tiny$x_{d}$] (x9) {} -- 
    (9,0) node[label=340:\tiny{$x_{d+1}$}] (x10) {};
\draw[dashed] (x6) to (x7);
\draw (x3) -- (2.4,0.5) node {};
\draw (x3) -- (2,-0.5) node {};
\draw (x3) -- (1.6,0.5) node {};
\draw (2.4,0.5) -- (2.4,1) node {};
\draw (2,-0.5) -- (2,-1) node {};
\draw (1.6,0.5) -- (1.6,1) node {};
\draw (x4) -- (3,0.5) node {};
\draw (x4) -- (3,-0.5) node {};
\draw (3,0.5) -- (3,1) node {};
\draw (3,-0.5) -- (3,-1) node {};
\draw (3,1) -- (3,1.5) node {};
\draw (x5) -- (3.6,0.5)  node {};
\draw (x5) -- (3.6,-0.5)  node {};
\draw (x5) -- (4.4,-0.5) node {};
\draw (x5) -- (4.4,0.5)  node {};
\draw (3.6,0.5) -- (3.6,1) node {};
\draw (3.6,-0.5) -- (3.6,-1) node {};
\draw (4.4,-0.5) -- (4.4,-1) node {};
\draw (4.4,0.5) -- (4.4,1) node {};
\draw (x6) -- (4.6,0.5)  node {};
\draw (x6) -- (5,-0.5)  node {};
\draw (x6) -- (5.4,0.5)  node {};
\draw (4.6,0.5) -- (4.6,1) node {};
\draw  (5,-0.5) -- (5,-1) node {};
\draw  (5.4,0.5) -- (5.4,1) node {};
\draw (x7) -- (6,0.5)  node {};
\draw (x7) -- (6,-0.5)  node {};
\draw  (6,0.5) -- (6,1) node {};
\draw  (6,-0.5) -- (6,-1) node {};
\draw[vertex] (x8) -- (6.6,0.5)  node {};
\draw[vertex] (x8) -- (6.6,-0.5)  node {};
\draw[vertex] (x8) -- (7.4,-0.5)  node {};
\draw[vertex] (x8) -- (7.4,0.5)  node {};
\draw[vertex] (x8) -- (7,0.5) node {};
\draw[vertex] (x8) -- (7,-0.5)  node {};
\draw[vertex]  (6.6,0.5) -- (6.6,1) node {};
\draw[vertex]  (6.6,-0.5) -- (6.6,-1) node {};
\draw[vertex]  (7.4,-0.5) -- (7.4,-1) node {};
\draw[vertex]  (7.4,0.5) -- (7.4,1) node {};
\draw[vertex]  (7,0.5) -- (7,1) node {};
\draw[vertex]  (7,-0.5) -- (7,-1) node {};
\end{tikzpicture}
\caption{An example of a tree where every support vertex has degree 2.}
\label{fig:1}
\end{figure}
\noindent Let $d\geq 5,$ and $Q':=\{z_{i_1},z_{i_2},\dots,z_{i_a}\}\cup \{ x_{d}\}$ be the set of support vertices adjacent to $x_{d-1}$, and $L':=\{y_{i_1},y_{i_2}\dots,y_{i_a}\}\cup \{ x_{d+1}\}$ be the corresponding pendant vertices. Let $\theta:=\deg_{T}(x_{d-1})$ and $\mu:=\deg_{T}(x_{d-2})$, the following isomorphism holds:
\begin{equation}\label{induced}
S/(I,x_{d-1})\cong K[V(T')]/I(T') \underset{j=1}{\overset{\theta-1}{\tensor_{K}}}K[V(P_2)]/I(P_2),
\end{equation}
where $T'$ is the induced subtree on the set of vertices $V(T)\backslash (Q'\cup L')$.  Using Lemma~\ref{circulentt} on Eq. \eqref{induced} gives \begin{equation}\label{eq1comma}
    \reg(S/(I,x_{d-1})) = \reg(K[V(T')]/I(T'))+\theta-1.
\end{equation}
By Lemma \ref{comareg}(a), $\reg(S/I)\geq \reg(S/(I,x_{d-1}))$, and using Eq. \eqref{eq1comma}, we have $$\reg(S/I)\geq \reg(S/(I,x_{d-1}))=\reg(K[V(T')]/I(T'))+\theta-1.$$ Hence, the proof reduces to establishing that $$\reg(K[V(T')]/I(T'))+\theta-1\geq \left\lfloor \frac{n-p+d+5}{6}\right\rfloor.$$

\noindent Now let $d \geq 5$. Let $Q' := \{z_{i_1}, z_{i_2}, \dots, z_{i_a}\} \cup \{x_d\}$ be the set of support vertices adjacent to $x_{d-1}$, and let $L' := \{y_{i_1}, y_{i_2}, \dots, y_{i_a}\} \cup \{x_{d+1}\}$ be the corresponding pendant vertices. Note that $x_{d-1}$ is also adjacent to $x_{d-2}$, so its degree is $\theta := \deg_T(x_{d-1}) = a + 2$. Similarly, let $\mu := \deg_T(x_{d-2})$.

Removing $x_{d-1}$ from $T$ disconnects the tree into several components: one is the induced subtree $T'$ on the vertex set $V(T) \setminus (Q' \cup L')$, and the others are $\theta-1$ disjoint edges, namely the edges incident to $x_{d-1}$ with vertices in $Q' \cup L'$. Specifically, these edges are $x_d x_{d+1}$ and $z_{i_j} y_{i_j}$ for $j = 1, \dots, a$. Each such edge is isomorphic to $P_2$. Therefore, we have the isomorphism:
\begin{equation}\label{induced}
S/(I, x_{d-1}) \cong K[V(T')]/I(T') \otimes_K \bigotimes_{j=1}^{\theta-1} K[V(P_2)]/I(P_2).
\end{equation}

Applying Lemma~\ref{circulentt} to equation \eqref{induced} and noting that $\reg(K[V(P_2)]/I(P_2)) = 1$, we obtain:
\begin{equation}\label{eq1comma}
\reg(S/(I, x_{d-1})) = \reg(K[V(T')]/I(T')) + \theta - 1.
\end{equation}

By Lemma~\ref{comareg}(a), we have $\reg(S/I) \geq \reg(S/(I, x_{d-1}))$. Combining this with equation \eqref{eq1comma} yields:
\[
\reg(S/I) \geq \reg(K[V(T')]/I(T')) + \theta - 1.
\]
Thus, to prove the theorem, it suffices to show that
\[
\reg(K[V(T')]/I(T')) + \theta - 1 \geq \left\lfloor \frac{n - p + d + 5}{6} \right\rfloor.
\]

Let $p'$ and $d'$ be the number of pendant vertices and diameter of $T'$. Observe that $d'\geq d-3$. This case is subdivided into four cases:\\
\textbf{a.} If $\theta=2$, and $\mu=2,$ then $|V(T')|=n-3$, and $p'=p.$ By induction on $n$ 
\[
\begin{aligned}
\operatorname{reg}(K[V(T')]/I(T')) + 1 &\geq \left\lfloor \frac{n - 3 - p + d' + 5}{6} \right\rfloor + 1 \\
&= \left\lfloor \frac{n - p + d' + 8}{6} \right\rfloor \\
&\geq \left\lfloor \frac{n - p + d + 5}{6} \right\rfloor.
\end{aligned}
\]
\textbf{b.} If $\theta=2$, and $\mu\geq3$, then $|V(T')|=n-3$, and $p'=p-1.$ By induction on $n$ 
\[
\begin{aligned}
\operatorname{reg}(K[V(T')]/I(T')) + 1 &\geq \left\lfloor \frac{n-3-(p-1)+d'+5}{6} \right\rfloor + 1 \\
&= \left\lfloor \frac{n-p+d'+9}{6} \right\rfloor \\
&\geq \left\lfloor \frac{n-p+d+5}{6} \right\rfloor.
\end{aligned}
\]
\textbf{c.} If $\theta\geq3$, and $\mu=2,$ then $|V(T')|=n-2\theta+1$, and $p'=p-\theta+2.$ By induction on $n$
\[
\begin{aligned}
\operatorname{reg}(K[V(T')]/I(T')) + \theta - 1 &\geq \left\lfloor \frac{n - 2\theta + 1 - (p - \theta + 2) + d' + 5}{6} \right\rfloor + \theta - 1 \\
&= \left\lfloor \frac{n - p + d' + 5\theta - 2}{6} \right\rfloor.
\end{aligned}
\]
Since $\theta\geq3,$ implies that $5\theta-2\geq13$. Therefore,
$$
\reg(K[V(T')]/I(T'))+\theta-1\geq  \left\lfloor \frac{n-p+d'+13}{6} \right\rfloor\geq  \left\lfloor \frac{n-p+d+5}{6} \right\rfloor.$$
\textbf{d.} If $\theta\geq3$, and $\mu\geq3,$ then $|V(T')|=n-2\theta+1$, and $p'=p-\theta+1.$ By induction on $n$
\[
\begin{aligned}
\operatorname{reg}(K[V(T')]/I(T')) + \theta - 1 &\geq \left\lfloor \frac{n - 2\theta + 1 - (p - \theta + 1) + d' + 5}{6} \right\rfloor + \theta - 1 \\
&= \left\lfloor \frac{n - p + d' + 5\theta - 1}{6} \right\rfloor.
\end{aligned} 
\]
Since $\theta\geq3,$ implies that $5\theta-1\geq14$. Therefore,
$$
\reg(K[V(T')]/I(T'))+\theta-1\geq  \left\lfloor \frac{n-p+d'+14}{6} \right\rfloor\geq  \left\lfloor \frac{n-p+d+5}{6} \right\rfloor.$$
\textbf{Case 2.} Now consider the case when $T$ admits a support vertex, say $x_u$, of degree $\geq 3$. Let $x_v$ be any of the pendant vertex adjacent to $x_u$. We have the following isomorphism:
$$ S/(I,x_{v})\cong K[V(T'')]/I(T''),$$ 
where $T''$ is an induced subtree on vertex set $V(T)\backslash\{x_v\}$. Let $p''$ be the number of pendants and $d''$ be the diameter of $T''$. Then $|V(T'')|=n-1$, $p''=p-1$, and $d''=d$. Applying Lemma \ref{comareg}(a) and  induction on $n$, we get
\[
\begin{aligned}
\operatorname{reg}(S/I) \geq \operatorname{reg}(S/(I, x_u))&= \operatorname{reg}(K[V(T'')]/I(T'')) \\
&\geq \left\lfloor \frac{n - 1 - (p - 1) + d + 5}{6} \right\rfloor \\
&= \left\lfloor \frac{n - p + d + 5}{6} \right\rfloor.
\end{aligned}
\]
\end{proof}

\begin{Theorem}\label{Theorem.UB-reg,tree}
Let $T$ be a tree of order \( n \geq 2 \) and pendant vertices $p$. If $S=K[V(T)]$ and $I=I(T)$, then
\[\reg(S/I) \leq \min \left\{n-p , \left\lfloor \frac{2n-p}{3} \right\rfloor\right\}.\]
\end{Theorem}
\begin{proof} If $n=2$, then $T \cong \mathcal{S}_1$ with $p=1$, and the result holds. For \( 3 \leq n \leq 6 \), the result can be verified using Lemma \ref{reg1}. For \( n = 7 \), the required inequality holds for all non-isomorphic trees, see Table \ref{All7VertexTrees}.
 Assume \( n \geq 8\). We first address the following special cases:\\
\textbf{i.} If \( d = 2, \)  then \( T \cong \mathcal{S}_p\) and $p=n-1$. By Lemma \ref{reg1}, $\reg(S/I)=1.$ The result is evident as: $$\min \left\{n-p , \left\lfloor \frac{2n-p}{3} \right\rfloor\right\} =\min \left\{n-(n-1) , \left\lfloor \frac{2n-(n-1)}{3} \right\rfloor\right\} =\min \left\{1 , \left\lfloor \frac{n+1}{3} \right\rfloor\right\}=1.$$
\textbf{ii.}  If \( d = 3, \)  then \( T \cong \mathcal{B}_p \) and \( p=n-2. \)  By Lemma \ref{reg1}, $\reg(S/I)=1.$ Again, the result is evident as: $$\min \left\{n-p , \left\lfloor \frac{2n-p}{3} \right\rfloor\right\} =\min \left\{n-(n-2) , \left\lfloor \frac{2n-(n-2)}{3} \right\rfloor\right\} =\min \left\{2 , \left\lfloor \frac{n+2}{3} \right\rfloor\right\}\geq 1.$$
Now assume $d \geq 4$ and $p \geq 3.$ Let $P_{d+1}$ be an induced path that realizes the diameter in $T$, and label the vertices of $P_{d+1}$ by $x_1,x_2,\ldots,x_{d+1}$ (where $x_i$ is adjacent to $x_{i+1}$ for all $1\leq i\leq d$). Clearly, $x_{d+1}$ is a pendant vertex and $x_d$ is a support vertex. 
We get the following isomorphisms:
$$ S/(I,x_{d+1})\cong K[V(T')]/I(T'), \text{ and } S/(I:x_{d+1})\cong K[V(T'')]/I(T'') {\tensor_{K}}K[N_T(x_d) \backslash \{x_{d-1}\}], $$ 
where $T'$ and $T''$ are the induced subtrees on the vertex sets $V(T')=V(T)\backslash \{x_{d+1}\}$ and $V(T'')=V(T)\backslash (N_T[x_d]\backslash \{x_{d-1}\})$, respectively. 
\begin{equation*}
 \reg(S/(I,x_{d+1})) = \reg(K[V(T')]/I(T')) \text{ and }   \reg(S/(I:x_{d+1})) = \reg(K[V(T'')]/I(T'')).
\end{equation*}
By Lemma \ref{comareg}(b), 
\begin{align}
\reg (S/I) & \leq \max\{\reg (S/(I:x_{d+1}))+1,\reg (S/(I,x_{d+1}))\}\notag \\
& =\max\{ \reg(K[V(T'')]/I(T''))+1,\reg(K[V(T')]/I(T'))\}.
\label{eq2comma}
\end{align}
Let $p',d'$ and $p'',d''$ denotes the number of pendants and diameter in $T'$ and $T''$, respectively. Note that $d'\geq d-1$ and $d'' \geq d - 2$. 
\noindent Let $\theta:=\deg_T(x_{d-1})$ and $\eta:=\deg_T(x_{d})$. The proof is divided into four cases:\\
\textbf{a.}   If $\theta=2,$ and $\eta=2$,  then $|V(T')|=n-1$, $|V(T'')|=n-2$, and $p'=p=p''.$ By induction on $n$ $$\reg(K[V(T')]/I(T'))\leq n-1-p,  $$and$$\reg(K[V(T'')]/I(T'')) \leq n-2-p.$$
Using Eq. \eqref{eq2comma}
 $$\reg (S/I)\leq \max\{n-2-p+1,n-1-p\}\leq n-p.$$
Similarly, $$\reg(K[V(T')]/I(T'))\leq  \left\lfloor \frac{2(n-1)-p}{3} \right\rfloor, $$and$$
 \reg(K[V(T'')]/I(T'')) \leq  \left\lfloor \frac{2(n-2)-p}{3} \right\rfloor.
$$
Using Eq. \eqref{eq2comma}
 $$\reg (S/I)\leq \max\left\{\left\lfloor \frac{2n-p-4}{3} \right\rfloor+1,\left\lfloor \frac{2n-p-2}{3} \right\rfloor\right\}\leq \left\lfloor \frac{2n-p}{3} \right\rfloor.$$
\textbf{b.}  If $\theta \geq3,$  and $\eta=2$, then $|V(T')|=n-1$, $|V(T'')|=n-2$, $p'=p$, and $p''=p-1.$ By induction on $n$ $$\reg(K[V(T')]/I(T'))\leq n-1-p, $$and$$ \reg(K[V(T'')]/I(T'')) \leq n-2-(p-1).$$
Using Eq. \eqref{eq2comma} $$\reg (S/I)\leq \max\{n-p,n-p-1\}=n-p.$$
Similarly,
$$
\reg(K[V(T')]/I(T'))\leq  \left\lfloor \frac{2(n-1)-p}{3} \right\rfloor, $$ and $$  \reg(K[V(T'')]/I(T'')) \leq  \left\lfloor \frac{2(n-2)-(p-1)}{3}\right\rfloor.
$$
Using Eq. \eqref{eq2comma} 
 $$\reg (S/I)\leq \max\left\{\left\lfloor \frac{2n-p-3}{3} \right\rfloor+1,\left\lfloor \frac{2n-p-2}{3} \right\rfloor\right\}= \left\lfloor \frac{2n-p}{3} \right\rfloor.$$
\textbf{c.}    If $\theta=2,$ and $\eta\geq3$, then $|V(T')|=n-1$, $|V(T'')|=n-\eta$, $p'=p-1$, and $p''=p-\eta+2.$ By induction on $n$ $$\reg(K[V(T')]/I(T'))\leq n-1-(p-1), $$ and $$ \reg(K[V(T'')]/I(T'')) \leq n-\eta-(p-\eta+2).$$
Using Eq. \eqref{eq2comma}
 $$\reg (S/I)\leq \max\{n-p-2+1,n-p\}=n-p.$$
 Similarly,
\[
\operatorname{reg}(K[V(T')]/I(T')) \leq \left\lfloor \frac{2(n-1)-(p-1)}{3} \right\rfloor,
\]
and
\[
\begin{aligned}
\operatorname{reg}(K[V(T'')]/I(T'')) &\leq \left\lfloor \frac{2(n-\eta)-(p-\eta+2)}{3} \right\rfloor \\
&= \left\lfloor \frac{2n-p-(\eta+2)}{3} \right\rfloor \\
&\leq \left\lfloor \frac{2n-p-4}{3} \right\rfloor.
\end{aligned}
\]
Using Eq. \eqref{eq2comma}
  $$\reg (S/I)\leq \max\left\{\left\lfloor \frac{2n-p-4}{3} \right\rfloor+1,\left\lfloor \frac{2n-p-1}{3} \right\rfloor\right\}\leq \left\lfloor \frac{2n-p}{3} \right\rfloor.$$
\textbf{d.}    If $\theta\geq3,$  and $\eta\geq3$, then $|V(T')|=n-1$, $|V(T'')|=n-\eta$, $p'=p-1$, and $p''=p-\eta+1.$ By induction on $n$ $$\reg(K[V(T')]/I(T'))\leq n-1-(p-1), $$ and $$ \reg(K[V(T'')]/I(T''))  \leq n-\eta-(p-\eta+1).$$
Using Eq. \eqref{eq2comma}
 $$\reg (S/I)\leq \max\{n-p,n-p\}=n-p.$$
Similarly,
\[
\operatorname{reg}(K[V(T')]/I(T')) \leq \left\lfloor \frac{2(n-1)-(p-1)}{3} \right\rfloor,
\]
and
\[
\begin{aligned}
\operatorname{reg}(K[V(T'')]/I(T'')) &\leq \left\lfloor \frac{2(n-\eta)-(p-\eta+1)}{3} \right\rfloor \\
&= \left\lfloor \frac{2n-p-(\eta+1)}{3} \right\rfloor \\
&\leq \left\lfloor \frac{2n-p-3}{3} \right\rfloor.
\end{aligned}
\]
Using Eq. \eqref{eq2comma}
 $$\reg (S/I)\leq \max\left\{\left\lfloor \frac{2n-p-3}{3} \right\rfloor+1,\left\lfloor \frac{2n-p-1}{3} \right\rfloor\right\}\leq \left\lfloor \frac{2n-p}{3} \right\rfloor.$$
\end{proof}
\begin{table}[h]
\centering
\renewcommand{\arraystretch}{0.01}
\setlength{\tabcolsep}{8pt}
\begin{tabular}{c c c c c}
\toprule
{Sr. No.} & {Tree (T)} & 
${\left\lfloor \frac{n-p+d+5}{6} \right\rfloor}$ & $\reg(K[V(T)]/I(T))$ & 
${\min\left\{n-p, \left\lfloor \frac{2n-p}{3} \right\rfloor\right\}}$ \\
\midrule
1.  & 
\begin{tikzpicture}[scale=0.5, baseline=-0.5ex]
\draw (0,0) -- (1,0) -- (2,0) -- (3,0) -- (4,0) -- (5,0) -- (6,0);
\foreach \x in {0,...,6} \filldraw (\x,0) circle (2pt);
\end{tikzpicture}
& 2 & 2 & 4 \\[2pt]
2. & 
\begin{tikzpicture}[scale=0.5, baseline=-0.5ex]
\draw (1,0) -- (0,0);
\draw (1,0) -- (2,0) -- (3,0) -- (4,0) -- (5,0);
\draw (1,0) -- (1,1);
\foreach \x in {0,1,2,3,4,5} \filldraw (\x,0) circle (2pt);
\filldraw (1,1) circle (2pt);
\end{tikzpicture}
& 2 & 2 & 3 \\[2pt]
3. & 
\begin{tikzpicture}[scale=0.5, baseline=-0.5ex]
\draw (2,0) -- (1,0) -- (0,0);
\draw (2,0) -- (3,0) -- (4,0) -- (5,0);
\draw (2,0) -- (2,1);
\foreach \x in {0,1,2,3,4,5} \filldraw (\x,0) circle (2pt);
\filldraw (2,1) circle (2pt);
\end{tikzpicture}
& 2 & 2 & 3 \\[2pt]
4. & 
\begin{tikzpicture}[scale=0.5, baseline=-0.5ex]
\draw (1,0) -- (0,0);
\draw (1,0) -- (2,0) -- (3,0) -- (4,0);
\draw (1,0) -- (1,1);
\draw (2,0) -- (2,1);
\foreach \x in {0,1,2,3,4} \filldraw (\x,0) circle (2pt);
\foreach \y in {1,2} \filldraw (\y,1) circle (2pt);
\end{tikzpicture}
& 2 & 2 & 3 \\[2pt]
5. & 
\begin{tikzpicture}[scale=0.5, baseline=-0.5ex]
\draw (1,0) -- (0,0);
\draw (1,0) -- (2,0) -- (3,0) -- (4,0);
\draw (2,0) -- (2.5,0.8);
\draw (2,0) -- (1.5,0.8);
\foreach \x in {0,1,2,3,4} \filldraw (\x,0) circle (2pt);
\filldraw (2.5,0.8) circle (2pt);
\filldraw (1.5,0.8) circle (2pt);
\end{tikzpicture}
& 2 & 2 & 3 \\[2pt]
6. & 
\begin{tikzpicture}[scale=0.5, baseline=-0.5ex]
\draw (1,0) -- (0,0);
\draw (1,0) -- (2,0);
\draw (1,0) -- (0.5,0.8);
\draw (1,0) -- (1.5,0.8);
\draw (2,0) -- (3,0);
\draw (4,0) -- (3,0);
\foreach \x in {0,1,2,3,4} \filldraw (\x,0) circle (2pt);
\filldraw (0.5,0.8) circle (2pt);
\filldraw (1.5,0.8) circle (2pt);
\end{tikzpicture}
& 2 & 2 & 3 \\[2pt]
7. & 
\begin{tikzpicture}[scale=0.5, baseline=-0.5ex]
\draw (1,0) -- (0,0);
\draw (1,0) -- (2,0) -- (3,0);
\draw (1,0) -- (0.5,0.8);
\draw (2,0) -- (2,1);
\draw (1,0) -- (1.5,0.8);
\foreach \x in {0,1,2,3} \filldraw (\x,0) circle (2pt);
\filldraw (2,1) circle (2pt);
\filldraw (0.5,0.8) circle (2pt);
\filldraw (1.5,0.8) circle (2pt);
\end{tikzpicture}
& 1 & 1 & 2 \\[2pt]
8. & 
\begin{tikzpicture}[scale=0.5, baseline=-0.5ex]
\draw (1,0) -- (0,0);
\draw (1,0) -- (1.8,-0.5) -- (2.8,-0.5);
\draw (1,0) -- (1.8,0.5);
\draw (1,0) -- (-1,0);
\draw (1.8,0.5) -- (2.8,0.5);
\foreach \x in {0,1} \filldraw (\x,0) circle (2pt);
\filldraw (2.8,0.5) circle (2pt);
\filldraw (-1,0) circle (2pt);
\filldraw (1.8,0.5) circle (2pt);
\filldraw (2.8,-0.5) circle (2pt);
\filldraw (1.8,-0.5) circle (2pt);
\end{tikzpicture}
& 2 & 3 & 3 \\[2pt]
9. & 
\begin{tikzpicture}[scale=0.5, baseline=-0.5ex]
\draw (0,0) -- (-0.5,0.5);
\draw (0,0) -- (0.5,0.5);
\draw (0,0) -- (-0.5,-0.5);
\draw (0,0) -- (0.5,-0.5);
\draw (0,0) -- (1,0);
\draw (1,0) -- (2,0);
\foreach \x in {0,1,2} \filldraw (\x,0) circle (2pt);
\foreach \y in {0.5,-0.5} {
    \filldraw (-0.5,\y) circle (2pt);
    \filldraw (0.5,\y) circle (2pt);
}
\end{tikzpicture}
& 1 & 1 & 2 \\[2pt]
10. & 
\begin{tikzpicture}[scale=0.5, baseline=-0.5ex]
\draw (1,0) -- (0.2,-0.5);
\draw (1,0) -- (2,0) -- (3,0) -- (3.8,0.5) ;
\draw (1,0) -- (0.2,0.5);
\draw (3,0) -- (3.8,-0.5) ;
\foreach \x in {1,2,3} \filldraw (\x,0) circle (2pt);
\filldraw (0.2,-0.5) circle (2pt);
\filldraw (0.2,0.5) circle (2pt);
\filldraw (3.8,-0.5) circle (2pt);
\filldraw (3.8,0.5) circle (2pt);
\end{tikzpicture}
& 2 & 2 & 3 \\[2pt]
11.  & 
\begin{tikzpicture}[scale=0.5, baseline=-0.5ex]
\filldraw (0,0) circle (2pt);
\foreach \angle in {0,60,120,180,240,300} {
    \draw (0,0) -- (\angle:1);
    \filldraw (\angle:1) circle (2pt);
}
\end{tikzpicture}
& 1 & 1 & 1 \\
\bottomrule
\end{tabular}
\caption{All non-isomorphic trees of order 7.}
\label{All7VertexTrees}
\end{table}
\begin{Remark}
{\em
The upper bound in Theorem~\ref{Theorem.UB-reg,tree} is determined by the minimum of two expressions involving the number of pendant vertices $p$ relative to the order $n$. The transition between these terms occurs at a critical threshold: for trees with $p < \lfloor n/2 \rfloor$ pendant vertices, the bound $\lfloor (2n-p)/3 \rfloor$ dominates, while for trees with $p > \lfloor n/2 \rfloor$, the bound $n-p$ becomes the determining factor. This threshold behavior, illustrated in Table~\ref{tab:reg_bounds} for trees of order 100, reveals how the bound adapts specifically to the pendant vertex distribution.}
\end{Remark}

\begin{table}[h]
\centering
\begin{tabular}{|c|c|c|}
\hline
{pendants (\( p \))} & \textbf{\( \reg(S/I(T)) \leq n - p \)} & \textbf{\( \reg(S/I(T)) \leq \lfloor \frac{2n - p}{3} \rfloor \)} \\
\hline
2 & 98 & {66} \\
5 & 95 & {65} \\
10 & 90 & {63} \\
20 & 80 & {60} \\
30 & 70 & {56} \\
40 & 60 & {53} \\
50 & 50 & 50 \\
60 &  {40} & 46 \\
70 &  {30} & 43 \\
80 &  {20} & 40 \\
90 &  {10} & 36 \\
95 &  {5} & 35 \\
99 &  {1} & 33 \\
\hline
\end{tabular}
\caption{Comparison of upper bounds for   \( n = 100 \).}
\label{tab:reg_bounds}
\end{table}


As a remarkable consequence, we establish explicit combinatorial bounds on the induced matching number for trees.

\begin{Corollary}\label{cor1}
Let $T$ be a tree of order $n \geq 2$, diameter $d$, and pendant vertices $p$. Then
\[ 
\left\lfloor \frac{n-p+d+5}{6} \right\rfloor \leq \operatorname{indmat}(T) \leq \min\left\{n-p, \left\lfloor \frac{2n-p}{3} \right\rfloor\right\}.
\]
\end{Corollary}
\begin{proof}
The equality follows from Lemma~\ref{reg1}, since $T$ is chordal. The bounds then follow by applying Theorems~\ref{Theorem.LB-reg,tree} and~\ref{Theorem.UB-reg,tree}
\end{proof}
\section{castelnuovo-mumford regularity of the edge ideals of multi-whisker trees}\label{Section 4} 
This section presents bounds for the Castelnuovo-Mumford regularity of edge ideals of multi-whisker trees. For a multi-whisker tree $T_{\mathbf{a}}$ with $\mathbf{a} = (a_1, \dots, a_n)$ where $a_i \in \mathbb{Z}_+^n$, we have the leverage of chordal structure of trees to connect Castelnuovo-Mumford regularity with combinatorial invariants. Since $T_{\mathbf{a}}$ is chordal, Lemma~\ref{reg1}, and Lemma~\ref{independence} gives $\operatorname{reg}(S/I(T_{\mathbf{a}})) = \operatorname{im}(T_{\mathbf{a}})=\alpha(T)$, 
leading to the fundamental bounds
\[
\left\lceil \frac{n}{2} \right\rceil \leq \operatorname{reg}(S/I(T_{\mathbf{a}})) \leq n-1.
\]

The multi-whisker construction enables sharper estimates. The main result of this section provides an improved upper bound that exploits the specific structure of multi-whisker trees. We begin with a straightforward lemma.

\begin{Lemma}\label{whisker-multiwhisker}
Let $G$ be a graph of order $n\geq2$. Then
\[\reg\left(K[V(G_{\mathbf{1}})]/I(G_{\mathbf{1}})\right) =\reg\left(K[V(G_{\mathbf{a}})]/I(G_{\mathbf{a}})\right).\]
\end{Lemma}
\begin{proof}
By Lemma~\ref{independence}, $\indmat(G_{\mathbf{a}}) = \alpha(G),$ for any vector $\mathbf{a}$. Considering the case where $a_i = 1$, for all $i=1,\dots,n$, gives
\(
\indmat(G_{\mathbf{1}}) = \alpha(G).
\)
Hence,  
\(
\indmat(G_{\mathbf{1}}) = \indmat(G_{\mathbf{a}})
,\)
and the result follows from Lemma~\ref{reg1}.
\end{proof}
\begin{Theorem}\label{combine-UB-1Whisker-p-d}
Let $T$ be a tree of order \( n \geq 2 \), diameter $d$, and pendant vertices $p$. If $S = K[V(T_{\mathbf{1}})]$ and $I = I(T_{\mathbf{1}})$, then
 $$\reg{(S/I)}\leq \min{\left\{\left\lceil\frac{2n-d-1}{2}\right\rceil,\left\lfloor\frac{2n+p-2}{3}\right\rfloor\right\}}.$$ 
\end{Theorem}
\begin{proof} 
For \( 2 \leq n \leq 5 \), the result follows by Lemma \ref{reg1}, see Table \ref{Table2}. Consider \( n \geq 6\). We first address the following special cases:\\
\textbf{i.} If \( d = 2, \)  then \( T_{\bold{1}} \cong (\mathcal{S}_p)_{\bold{1}}\), and $p=n-1$. By Lemma \ref{reg1}, $\reg(S/I)=n-1.$ The result follows because:
\begin{align}
\min{\left\{\left\lceil\frac{2n-d-1}{2}\right\rceil,\left\lfloor\frac{2n+p-2}{3}\right\rfloor\right\}} & =\min{\left\{\left\lceil\frac{2n-2-1}{2}\right\rceil,\left\lfloor\frac{2n+n-1-2}{3}\right\rfloor\right\}}\notag \\
& =\min{\left\{\left\lceil\frac{2n-3}{2}\right\rceil,n-1\right\}}=n-1.\notag
\end{align}
\textbf{ii.} If \( d = 3, \)  then \( T_{\bold{1}} \cong (\mathcal{B}_p)_{\bold{1}}, \) and \( p=n-2. \) By Lemma \ref{reg1}, $\reg(S/I)= n-2.$ Again, the result follows because: 
\begin{align}
\min{\left\{\left\lceil\frac{2n-d-1}{2}\right\rceil,\left\lfloor\frac{2n+p-2}{3}\right\rfloor\right\}} & =\min{\left\{\left\lceil\frac{2n-3-1}{2}\right\rceil,\left\lfloor\frac{2n+n-2-2}{3}\right\rfloor\right\}}\notag \\
& =\min{\left\{\left\lceil\frac{2n-4}{2}\right\rceil,\left\lfloor\frac{3n-4}{3}\right\rfloor\right\}}=n-2.\notag
\end{align} 
Consider $d \geq 4.$ By the definition of $T_{\bold{1}}$, the tree $T$ can be seen as an induced subtree of $T_{\bold{1}}$.  Let $P_{d+1}$ be an induced path that realizes the diameter in $T$, and label the vertices of $P_{d+1}$ by $x_1,x_2,\ldots,x_{d+1}$ (where $x_i$ is adjacent to $x_{i+1}$ for all $1\leq i\leq d$). 
Let $V(T)=\{x_1,x_2,\ldots,x_{d+1},\dots,x_n\}$. We label the corresponding whiskers in $T_{\bold{1}}$ by $y_1,y_2,\ldots,y_{d+1},\dots,y_n$ (that is, $y_i$ is adjacent to $x_{i}$ for all $i$), and
$V(T_{\bold{1}})=\{x_1,x_2,\ldots,x_{d+1},\dots,x_n,y_1,y_2,\ldots,y_{d+1},\dots,y_n\}$.
Let $\theta:=\deg_T(x_{d-1})$ and $\eta:=\deg_T(x_{d})$. Let $x_{r_1},\dots,x_{r_{\eta-2}},x_{d-1},x_{d+1}$ be the vertices adjacent to $x_d$ in $T$, and
$y_{r_1},\dots,y_{r_{\eta-2}},y_{d-1},y_{d+1}$ be the corresponding whiskers in $T_{\bold{1}}$. We have the following isomorphisms:
\begin{align}
S/(I,x_{d+1}) &\cong K[V(T'_{\bold{1}})]/I(T'_{\bold{1}}) {\tensor_{K}}K[y_{d+1}],\label{Eq4.1} \\
 S/(I:x_{d+1}) &\cong K[V(T''_{\bold{1}})]/I(T''_{\bold{1}}) \underset{j=1}{\overset{\eta-2}{\tensor_{K}}}K[V(P_2)]/I(P_2) {\tensor_{K}}K[x_{d+1}],\label{Eq4.2}
\end{align}
where $T'_{\bold{1}}$ and $T''_{\bold{1}}$ be the induced subtrees of $T_{\bold{1}}$ on the vertex sets $V(T'_{\bold{1}})=V(T_{\bold{1}})\backslash \{x_{d+1},y_{d+1}\}$ and $V(T''_{\bold{1}})=V(T_{\bold{1}})\backslash \{x_{r_1},\dots,x_{r_{\eta-2}},x_d,x_{d+1},y_{r_1},\dots,y_{r_{\eta-2}},y_d,y_{d+1}\}$, respectively. 
While $T'$ and $T''$ are the induced subtrees of $T$ (after removing whiskers from $T'_{\bold{1}}$ and $T''_{\bold{1}}$) on the vertex sets $V(T')=V(T)\backslash \{x_{d+1}\}$ and $V(T'')=V(T)\backslash \{N_T[x_d]\backslash \{x_{d-1}\}\}$, respectively. 
Let $p',d'$ and $p'',d''$ denotes the number of pendants and diameter in $T'$ and $T''$, respectively. Note that $d'\geq d-1$ and $d'' \geq d - 2$. Applying Lemma \ref{circulentt} to Eqs. \eqref{Eq4.1} and \eqref{Eq4.2}, one has 
\begin{equation*}
 \reg(S/(I,x_{d+1})) = \reg(K[V(T'_{\bold{1}})]/I(T'_{\bold{1}})) \text{ and }   \reg(S/(I:x_{d+1})) = \reg(K[V(T''_{\bold{1}})]/I(T''_{\bold{1}}))+\eta-2.
\end{equation*}
Now by Lemma \ref{comareg}(b),
\begin{align}
\reg (S/I) & \leq \max\{\reg (S/(I:x_{d+1}))+1,\reg (S/(I,x_{d+1}))\}\notag \\
& =\max\{ \reg(K[V(T''_{\bold{1}})]/I(T''_{\bold{1}}))+\eta-1,\reg(K[V(T'_{\bold{1}})]/I(T'_{\bold{1}}))\}.
\label{eq4max}
\end{align}
\noindent The proof is divided into four cases:\\
\textbf{a.}   If $\theta=2,$ and $\eta=2$,  then, $|V(T')|=n-1$, $|V(T'')|=n-2$, and $p'=p=p''.$ By induction on $n$
\[
\begin{aligned}
\operatorname{reg}(K[V(T'_{\mathbf{1}})]/I(T'_{\mathbf{1}})) &\leq \left\lceil\frac{2(n-1)-d'-1}{2}\right\rceil \\
&= \left\lceil\frac{2n-d'-3}{2}\right\rceil \\
&\leq \left\lceil\frac{2n-d-2}{2}\right\rceil,
\end{aligned}
\]
and
\[
\begin{aligned}
\operatorname{reg}(K[V(T''_{\mathbf{1}})]/I(T''_{\mathbf{1}})) + 1 &\leq \left\lceil\frac{2(n-2)-d''-1}{2}\right\rceil + 1 \\
&= \left\lceil\frac{2n-d''-3}{2}\right\rceil \\
&\leq \left\lceil\frac{2n-d-1}{2}\right\rceil.
\end{aligned}
\]
Using Eq. \eqref{eq4max}
 $$\reg (S/I)\leq \max\left\{\left\lceil\frac{2n-d-1}{2}\right\rceil,\left\lceil\frac{2n-d-2}{2}\right\rceil\right\}=\left\lceil\frac{2n-d-1}{2}\right\rceil.$$
Similarly, 
\[
\begin{aligned}
\operatorname{reg}(K[V(T'_{\mathbf{1}})]/I(T'_{\mathbf{1}})) &\leq \left\lfloor \frac{2(n-1)+p-2}{3} \right\rfloor \\
&= \left\lfloor \frac{2n+p-4}{3} \right\rfloor,
\end{aligned}
\]
and
\[
\begin{aligned}
\operatorname{reg}(K[V(T''_{\mathbf{1}})]/I(T''_{\mathbf{1}})) + 1 &\leq \left\lfloor \frac{2(n-2)+p-2}{3} \right\rfloor + 1 \\
&= \left\lfloor \frac{2n+p-3}{3} \right\rfloor.
\end{aligned}
\]
Using Eq. \eqref{eq4max}
 $$\reg (S/I)\leq \max\left\{\left\lfloor \frac{2n+p-4}{3} \right\rfloor,\left\lfloor \frac{2n+p-3}{3} \right\rfloor\right\}\leq \left\lfloor \frac{2n+p-2}{3} \right\rfloor.$$
\textbf{b.}  If $\theta \geq3,$  and $\eta=2$, then, $|V(T')|=n-1$, $|V(T'')|=n-2$, $p'=p$, and $p''=p-1.$ By induction on $n$ 
\[
\begin{aligned}
\operatorname{reg}(K[V(T'_{\mathbf{1}})]/I(T'_{\mathbf{1}})) &\leq \left\lceil\frac{2(n-1)-d'-1}{2}\right\rceil \\
&= \left\lceil\frac{2n-d'-3}{2}\right\rceil \\
&\leq \left\lceil\frac{2n-d-2}{2}\right\rceil,
\end{aligned}
\]
and
\[
\begin{aligned}
\operatorname{reg}(K[V(T''_{\mathbf{1}})]/I(T''_{\mathbf{1}})) + 1 &\leq \left\lceil\frac{2(n-2)-d''-1}{2}\right\rceil + 1 \\
&= \left\lceil\frac{2n-d''-3}{2}\right\rceil \\
&\leq \left\lceil\frac{2n-d-1}{2}\right\rceil.
\end{aligned}
\]
Using Eq. \eqref{eq4max}
 $$\reg (S/I)\leq \max\left\{\left\lceil\frac{2n-d-1}{2}\right\rceil,\left\lceil\frac{2n-d-2}{2}\right\rceil\right\}=\left\lceil\frac{2n-d-1}{2}\right\rceil.$$
Similarly, 
\[
\begin{aligned}
\operatorname{reg}(K[V(T'_{\mathbf{1}})]/I(T'_{\mathbf{1}})) &\leq \left\lfloor \frac{2(n-1)+p-2}{3} \right\rfloor \\
&= \left\lfloor \frac{2n+p-4}{3} \right\rfloor,
\end{aligned}
\]
and
\[
\begin{aligned}
\operatorname{reg}(K[V(T''_{\mathbf{1}})]/I(T''_{\mathbf{1}})) + 1 &\leq \left\lfloor \frac{2(n-2)+(p-1)-2}{3} \right\rfloor + 1 \\
&= \left\lfloor \frac{2n+p-4}{3} \right\rfloor.
\end{aligned}
\]
Using Eq. \eqref{eq4max}
 $$\reg (S/I)\leq \max\left\{\left\lfloor \frac{2n+p-4}{3} \right\rfloor,\left\lfloor \frac{2n+p-4}{3} \right\rfloor\right\}\leq \left\lfloor \frac{2n+p-2}{3} \right\rfloor.$$
\textbf{c.}    If $\theta=2,$ and $\eta\geq3$, then, $|V(T')|=n-1$, $|V(T'')|=n-\eta$, $p'=p-1$, and $p''=p-\eta+2.$ By induction on $n$ 
\[
\begin{aligned}
\operatorname{reg}(K[V(T'_{\mathbf{1}})]/I(T'_{\mathbf{1}})) &\leq \left\lceil\frac{2(n-1)-d'-1}{2}\right\rceil \\
&= \left\lceil\frac{2n-d'-3}{2}\right\rceil \\
&\leq \left\lceil\frac{2n-d-2}{2}\right\rceil,
\end{aligned}
\]
and
\[
\begin{aligned}
\operatorname{reg}(K[V(T''_{\mathbf{1}})]/I(T''_{\mathbf{1}})) + \eta - 1 &\leq \left\lceil\frac{2(n-\eta)-d''-1}{2}\right\rceil + \eta - 1 \\
&= \left\lceil\frac{2n-d''-3}{2}\right\rceil \\
&\leq \left\lceil\frac{2n-d-1}{2}\right\rceil.
\end{aligned}
\]
Using Eq. \eqref{eq4max}
 $$\reg (S/I)\leq \max\left\{\left\lceil\frac{2n-d-1}{2}\right\rceil,\left\lceil\frac{2n-d-2}{2}\right\rceil\right\}=\left\lceil\frac{2n-d-1}{2}\right\rceil.$$
Similarly, 
\[
\begin{aligned}
\operatorname{reg}(K[V(T'_{\mathbf{1}})]/I(T'_{\mathbf{1}})) &\leq \left\lfloor \frac{2(n-1)+(p-1)-2}{3} \right\rfloor \\
&= \left\lfloor \frac{2n+p-5}{3} \right\rfloor,
\end{aligned}
\]
and
\[
\begin{aligned}
\operatorname{reg}(K[V(T''_{\mathbf{1}})]/I(T''_{\mathbf{1}})) + \eta - 1 &\leq \left\lfloor \frac{2(n-\eta)+(p-\eta+2)-2}{3} \right\rfloor + \eta - 1 \\
&= \left\lfloor \frac{2n+p-3}{3} \right\rfloor.
\end{aligned}
\]
Using Eq. \eqref{eq4max}
 $$\reg (S/I)\leq \max\left\{\left\lfloor \frac{2n+p-5}{3} \right\rfloor,\left\lfloor \frac{2n+p-3}{3} \right\rfloor\right\}\leq \left\lfloor \frac{2n+p-2}{3} \right\rfloor.$$
\textbf{d.}    If $\theta\geq3,$  and $\eta\geq3$, then, $|V(T')|=n-1$, $|V(T'')|=n-\eta$, $p'=p-1$, and $p''=p-\eta+1.$ By induction on $n$ 
\[
\begin{aligned}
\operatorname{reg}(K[V(T'_{\mathbf{1}})]/I(T'_{\mathbf{1}})) &\leq \left\lceil\frac{2(n-1)-d'-1}{2}\right\rceil \\
&= \left\lceil\frac{2n-d'-3}{2}\right\rceil \\
&\leq \left\lceil\frac{2n-d-2}{2}\right\rceil,
\end{aligned}
\]
and
\[
\begin{aligned}
\operatorname{reg}(K[V(T''_{\mathbf{1}})]/I(T''_{\mathbf{1}})) + \eta - 1 &\leq \left\lceil\frac{2(n-\eta)-d''-1}{2}\right\rceil + \eta - 1 \\
&= \left\lceil\frac{2n-d''-3}{2}\right\rceil \\
&\leq \left\lceil\frac{2n-d-1}{2}\right\rceil.
\end{aligned}
\]
Using Eq. \eqref{eq4max}
 $$\reg (S/I)\leq \max\left\{\left\lceil\frac{2n-d-1}{2}\right\rceil,\left\lceil\frac{2n-d-2}{2}\right\rceil\right\}=\left\lceil\frac{2n-d-1}{2}\right\rceil.$$
Similarly, 
\[
\begin{aligned}
\operatorname{reg}(K[V(T'_{\mathbf{1}})]/I(T'_{\mathbf{1}})) &\leq \left\lfloor \frac{2(n-1)+(p-1)-2}{3} \right\rfloor \\
&= \left\lfloor \frac{2n+p-5}{3} \right\rfloor,
\end{aligned}
\]
and
\[
\begin{aligned}
\operatorname{reg}(K[V(T''_{\mathbf{1}})]/I(T''_{\mathbf{1}})) + \eta - 1 &\leq \left\lfloor \frac{2(n-\eta)+(p-\eta+1)-2}{3} \right\rfloor + \eta - 1 \\
&= \left\lfloor \frac{2n+p-4}{3} \right\rfloor.
\end{aligned}
\]
Using Eq. \eqref{eq4max}
 $$\reg (S/I)\leq \max\left\{\left\lfloor \frac{2n+p-5}{3} \right\rfloor,\left\lfloor \frac{2n+p-4}{3} \right\rfloor\right\}\leq \left\lfloor \frac{2n+p-2}{3} \right\rfloor.$$
\end{proof}

\begin{table}[h]
\centering
\renewcommand{\arraystretch}{0.01}
\setlength{\extrarowheight}{5pt}
\setlength{\tabcolsep}{4pt}
\begin{tabular}{c c c c c c c c}
\toprule
\footnotesize{Sr. No.}  & $\scriptstyle{T}$ & 
{$\scriptstyle{T_\bold{1}}$}  &
\footnotesize{$\,{n}\,$} & 
\footnotesize$\,{p}\,$ & 
\footnotesize$\,{d}\,$  &
$\scriptstyle{\reg\left(S/I(T_\bold{1})\right)}$ & 
$\scriptstyle{\min{\left\{\left\lceil\frac{2n-d-1}{2}\right\rceil,\left\lfloor\frac{2n+p-2}{3}\right\rfloor\right\}}}$ \\
\midrule
1. &

\begin{tikzpicture}[scale=0.5, baseline=-0.5ex]
\draw (0,0) -- (1,0);
\filldraw (0,0) circle (2pt);
\filldraw (1,0) circle (2pt);
\end{tikzpicture} &
\begin{tikzpicture}[scale=0.5, baseline=-0.5ex]
\draw (0,0) -- (1,0);
\filldraw (0,0) circle (2pt);
\filldraw (1,0) circle (2pt);
\draw[blue] (0,0) -- (-0.5,0.5);
\filldraw[blue] (-0.5,0.5) circle (2pt);
\draw[blue] (1,0) -- (1.5,0.5);
\filldraw[blue] (1.5,0.5) circle (2pt);
\end{tikzpicture}
 & 2 &  1 & 1 & 1 & 1 \\[1pt]

2. &
\begin{tikzpicture}[scale=0.5, baseline=-0.5ex]
\draw (0,0) -- (1,0) -- (2,0);
\filldraw (0,0) circle (2pt);
\filldraw (1,0) circle (2pt);
\filldraw (2,0) circle (2pt);
\end{tikzpicture}&
\begin{tikzpicture}[scale=0.5, baseline=-0.5ex]
\draw (0,0) -- (1,0) -- (2,0);
\filldraw (0,0) circle (2pt);
\filldraw (1,0) circle (2pt);
\filldraw (2,0) circle (2pt);
\draw[blue] (0,0) -- (-0.5,0.5);
\filldraw[blue] (-0.5,0.5) circle (2pt);
\draw[blue] (1,0) -- (1,0.5);
\filldraw[blue] (1,0.5) circle (2pt);
\draw[blue] (2,0) -- (2.5,0.5);
\filldraw[blue] (2.5,0.5) circle (2pt);
\end{tikzpicture}
 & 3 & 2 & 2 & 2& 2 \\[1pt]

3. &

\begin{tikzpicture}[scale=0.5, baseline=-0.5ex]
\draw (0,0) -- (1,0) -- (2,0) -- (3,0);
\filldraw (0,0) circle (2pt);
\filldraw (1,0) circle (2pt);
\filldraw (2,0) circle (2pt);
\filldraw (3,0) circle (2pt);

\end{tikzpicture}&

\begin{tikzpicture}[scale=0.5, baseline=-0.5ex]
\draw (0,0) -- (1,0) -- (2,0) -- (3,0);
\filldraw (0,0) circle (2pt);
\filldraw (1,0) circle (2pt);
\filldraw (2,0) circle (2pt);
\filldraw (3,0) circle (2pt);
\draw[blue] (0,0) -- (-0.5,0.5);
\filldraw[blue] (-0.5,0.5) circle (2pt);
\draw[blue] (1,0) -- (1,0.5);
\filldraw[blue] (1,0.5) circle (2pt);
\draw[blue] (2,0) -- (2,0.5);
\filldraw[blue] (2,0.5) circle (2pt);
\draw[blue] (3,0) -- (3.5,0.5);
\filldraw[blue] (3.5,0.5) circle (2pt);
\end{tikzpicture}
 & 4 &  2 & 3 & 2& 2 \\[1pt]

4. &

\begin{tikzpicture}[scale=0.5, baseline=-0.5ex]
\draw (1,0) -- (0,0);
\draw (1,0) -- (1,1);
\draw (1,0) -- (2,0);
\filldraw (0,0) circle (2pt);
\filldraw (1,0) circle (2pt);
\filldraw (1,1) circle (2pt);
\filldraw (2,0) circle (2pt);

\end{tikzpicture} &

\begin{tikzpicture}[scale=0.5, baseline=-0.5ex]
\draw (1,0) -- (0,0);
\draw (1,0) -- (1,1);
\draw (1,0) -- (2,0);
\filldraw (0,0) circle (2pt);
\filldraw (1,0) circle (2pt);
\filldraw (1,1) circle (2pt);
\filldraw (2,0) circle (2pt);
\draw[blue] (0,0) -- (-0.5,0.5);
\filldraw[blue] (-0.5,0.5) circle (2pt);
\draw[blue] (1,0) -- (0.5,0.5);
\filldraw[blue] (0.5,0.5) circle (2pt);
\draw[blue] (1,1) -- (0.5,1.5);
\filldraw[blue] (0.5,1.5) circle (2pt);
\draw[blue] (2,0) -- (2.5,0.5);
\filldraw[blue] (2.5,0.5) circle (2pt);
\end{tikzpicture}
 & 4 &  3 & 2 & 3& 3 \\[1pt]

5. &

\begin{tikzpicture}[scale=0.5, baseline=-0.5ex]
\draw (0,0) -- (1,0) -- (2,0) -- (3,0) -- (4,0);
\filldraw (0,0) circle (2pt);
\filldraw (1,0) circle (2pt);
\filldraw (2,0) circle (2pt);
\filldraw (3,0) circle (2pt);
\filldraw (4,0) circle (2pt);
\end{tikzpicture} &

\begin{tikzpicture}[scale=0.5, baseline=-0.5ex]
\draw (0,0) -- (1,0) -- (2,0) -- (3,0) -- (4,0);
\filldraw (0,0) circle (2pt);
\filldraw (1,0) circle (2pt);
\filldraw (2,0) circle (2pt);
\filldraw (3,0) circle (2pt);
\filldraw (4,0) circle (2pt);
\draw[blue] (0,0) -- (-0.5,0.5);
\filldraw[blue] (-0.5,0.5) circle (2pt);
\draw[blue] (1,0) -- (1,0.5);
\filldraw[blue] (1,0.5) circle (2pt);
\draw[blue] (2,0) -- (2,0.5);
\filldraw[blue] (2,0.5) circle (2pt);
\draw[blue] (3,0) -- (3,0.5);
\filldraw[blue] (3,0.5) circle (2pt);
\draw[blue] (4,0) -- (4.5,0.5);
\filldraw[blue] (4.5,0.5) circle (2pt);
\end{tikzpicture}
 & 5 &  2 & 4 & 3& 3 \\[1pt]

6. &

\begin{tikzpicture}[scale=0.5, baseline=-0.5ex]
\draw (1,0) -- (0,0);
\draw (1,0) -- (2,0);
\draw (1,0) -- (1,1);
\draw (2,0) -- (3,0);
\filldraw (0,0) circle (2pt);
\filldraw (1,0) circle (2pt);
\filldraw (1,1) circle (2pt);
\filldraw (2,0) circle (2pt);
\filldraw (3,0) circle (2pt);
\end{tikzpicture} &

\begin{tikzpicture}[scale=0.5, baseline=-0.5ex]
\draw (1,0) -- (0,0);
\draw (1,0) -- (2,0);
\draw (1,0) -- (1,1);
\draw (2,0) -- (3,0);
\filldraw (0,0) circle (2pt);
\filldraw (1,0) circle (2pt);
\filldraw (1,1) circle (2pt);
\filldraw (2,0) circle (2pt);
\filldraw (3,0) circle (2pt);
\draw[blue] (0,0) -- (-0.5,0.5);
\filldraw[blue] (-0.5,0.5) circle (2pt);
\draw[blue] (1,0) -- (0.5,0.5);
\filldraw[blue] (0.5,0.5) circle (2pt);
\draw[blue] (1,1) -- (0.5,1.5);
\filldraw[blue] (0.5,1.5) circle (2pt);
\draw[blue] (2,0) -- (2,0.5);
\filldraw[blue] (2,0.5) circle (2pt);
\draw[blue] (3,0) -- (3.5,0.5);
\filldraw[blue] (3.5,0.5) circle (2pt);
\end{tikzpicture}
 & 5 &  3 & 3 & 3& 3 \\[1pt]

7.&

\begin{tikzpicture}[scale=0.5, baseline=-0.5ex]
\draw (0.5,0) -- (-0.5,0);
\draw (0.5,0) -- (1.5,0);
\draw (0.5,0) -- (0.5,1);
\draw (0.5,0) -- (0.5,-1);
\filldraw (-0.5,0) circle (2pt);
\filldraw (1.5,0) circle (2pt);
\filldraw (0.5,1) circle (2pt);
\filldraw (0.5,-1) circle (2pt);
\filldraw (0.5,0) circle (2pt);
\end{tikzpicture}
 &

\begin{tikzpicture}[scale=0.5, baseline=-0.5ex]
\draw (0.5,0) -- (-0.5,0);
\draw (0.5,0) -- (1.5,0);
\draw (0.5,0) -- (0.5,1);
\draw (0.5,0) -- (0.5,-1);
\filldraw (-0.5,0) circle (2pt);
\filldraw (1.5,0) circle (2pt);
\filldraw (0.5,1) circle (2pt);
\filldraw (0.5,-1) circle (2pt);
\filldraw (0.5,0) circle (2pt);
\draw[blue] (-0.5,0) -- (-1,0.5);
\filldraw[blue] (-1,0.5) circle (2pt);
\draw[blue] (1.5,0) -- (2,0.5);
\filldraw[blue] (2,0.5) circle (2pt);
\draw[blue] (0.5,1) -- (0,1.5);
\filldraw[blue] (0,1.5) circle (2pt);
\draw[blue] (0.5,-1) -- (0,-0.5);
\filldraw[blue] (0,-0.5) circle (2pt);
\draw[blue] (0.5,0) -- (0,0.5);
\filldraw[blue] (0,0.5) circle (2pt);
\end{tikzpicture}
 & 5 &  4 & 2 & 4& 4 \\
\bottomrule
\end{tabular}
\caption{Comparison of $\reg(S/I(T_{\mathbf{1}}))$ and its upper bound given in Theorem \ref{combine-UB-1Whisker-p-d} for all non-isomorphic trees with $2 \leq n \leq 5$.}
\label{Table2}
\end{table}

The following result combines Lemma~\ref{reg1} and~\ref{independence} to obtain bounds for the independence number of trees.

\begin{Corollary}\label{cor,WT}
Let $T$ be a tree of order $n \geq 2$, diameter $d$, and pendant vertices $p$. Then
\[
\alpha(T) \leq \min\left\{\left\lceil\frac{2n-d-1}{2}\right\rceil, \left\lfloor\frac{2n+p-2}{3}\right\rfloor\right\}.
\]
\end{Corollary}

\begin{proof}
The equality follows immediately from Lemma~\ref{reg1} and~\ref{independence}, since $T_{\mathbf{a}}$ is chordal. The inequality then follows from Theorem~\ref{combine-UB-1Whisker-p-d}.
\end{proof}

\begin{Example}
{\em 
The upper bound in Corollary~\ref{cor,WT} is determined by the minimum of two combinatorial expressions that captures distinct structural aspects of trees. We illustrae this phenomenon by comparing the bounds for two non-isomorphic trees of order 9. As shown in Table~\ref{tab:reg-ub-comparison}, the diameter dependent term $\lceil(2n-d-1)/2\rceil$ governs the bound in one case, while the pendant vertex term $\lfloor(2n+p-2)/3\rfloor$ dominates in the other. This dichotomy demonstrates how the bound works, revealing its sensitivity to the combinatorial structure.
\begin{table}[h]
\centering
\renewcommand{\arraystretch}{1.5}
\begin{tabular}{c m{2.1cm} m{2.5cm} c c c c c}
\toprule
\footnotesize{Sr. No.}  & \,\,\,\,\,\,\footnotesize{$T$ with $n=9$} & \,\,\,\,\,\,\,\,\,\,\,\,\footnotesize{$T_{(2,\dots,2)}$} & \,\,\,\,\footnotesize{$p$} \,\,\,\, & \,\,\,\,\footnotesize{$d$}\,\,\,\, & \footnotesize{$\reg(S/I(T_\mathbf{a}))$} & \footnotesize{$\left\lceil\frac{2n-d-1}{2}\right\rceil$} & \footnotesize{$\left\lfloor\frac{2n+p-2}{3}\right\rfloor$} \\
\midrule
1. &
\begin{minipage}{2.3cm}
\centering
\scalebox{0.4}{
\begin{tikzpicture}[mainnode/.style={circle, draw, fill=black, inner sep=1.6pt},whiskernode/.style={circle, draw=blue, fill=blue, inner sep=1.6pt},whisker/.style={blue, line width=0.5pt}]
\node[mainnode] (c) at (0,0) {};
\node[mainnode] (n1) at (0,1) {};
\node[mainnode] (n2) at (0,2) {};
\node[mainnode] (s1) at (0,-1) {};
\node[mainnode] (s2) at (0,-2) {};
\node[mainnode] (e1) at (1,0) {};
\node[mainnode] (e2) at (2,0) {};
\node[mainnode] (w1) at (-1,0) {};
\node[mainnode] (w2) at (-2,0) {};
\draw (c) -- (n1) -- (n2);
\draw (c) -- (s1) -- (s2);
\draw (c) -- (e1) -- (e2);
\draw (c) -- (w1) -- (w2);
\end{tikzpicture}
}
\end{minipage} &
\begin{minipage}{2cm}
\centering
\scalebox{0.4}{
\begin{tikzpicture}[mainnode/.style={circle, draw, fill=black, inner sep=1.6pt},whiskernode/.style={circle, draw=blue, fill=blue, inner sep=1.6pt},whisker/.style={blue, line width=0.5pt}]
\node[mainnode] (c) at (0,0) {};
\node[mainnode] (n1) at (0,1) {};
\node[mainnode] (n2) at (0,2) {};
\node[mainnode] (s1) at (0,-1) {};
\node[mainnode] (s2) at (0,-2) {};
\node[mainnode] (e1) at (1,0) {};
\node[mainnode] (e2) at (2,0) {};
\node[mainnode] (w1) at (-1,0) {};
\node[mainnode] (w2) at (-2,0) {};
\draw (c) -- (n1) -- (n2);
\draw (c) -- (s1) -- (s2);
\draw (c) -- (e1) -- (e2);
\draw (c) -- (w1) -- (w2);
\draw[whisker] (c) -- (0.5,0.5);
\draw[whisker] (c) -- (-0.5,0.5);
\node[whiskernode] at (0.5,0.5) {};
\node[whiskernode] at (-0.5,0.5) {};
\draw[whisker] (n1) -- (0.5,1.5);
\draw[whisker] (n1) -- (-0.5,1.5);
\node[whiskernode] at (0.5,1.5) {};
\node[whiskernode] at (-0.5,1.5) {};
\draw[whisker] (n2) -- (0.5,2.5);
\draw[whisker] (n2) -- (-0.5,2.5);
\node[whiskernode] at (0.5,2.5) {};
\node[whiskernode] at (-0.5,2.5) {};
\draw[whisker] (s1) -- (0.5,-1.5);
\draw[whisker] (s1) -- (-0.5,-1.5);
\node[whiskernode] at (0.5,-1.5) {};
\node[whiskernode] at (-0.5,-1.5) {};
\draw[whisker] (s2) -- (0.5,-2.5);
\draw[whisker] (s2) -- (-0.5,-2.5);
\node[whiskernode] at (0.5,-2.5) {};
\node[whiskernode] at (-0.5,-2.5) {};
\draw[whisker] (e1) -- (1.5,0.5);
\draw[whisker] (e1) -- (1.5,-0.5);
\node[whiskernode] at (1.5,0.5) {};
\node[whiskernode] at (1.5,-0.5) {};
\draw[whisker] (e2) -- (2.5,0.5);
\draw[whisker] (e2) -- (2.5,-0.5);
\node[whiskernode] at (2.5,0.5) {};
\node[whiskernode] at (2.5,-0.5) {};
\draw[whisker] (w1) -- (-1.5,0.5);
\draw[whisker] (w1) -- (-1.5,-0.5);
\node[whiskernode] at (-1.5,0.5) {};
\node[whiskernode] at (-1.5,-0.5) {};
\draw[whisker] (w2) -- (-2.5,0.5);
\draw[whisker] (w2) -- (-2.5,-0.5);
\node[whiskernode] at (-2.5,0.5) {};
\node[whiskernode] at (-2.5,-0.5) {};
\end{tikzpicture}}
\end{minipage} & 
4 & 4 & 5 & 7 & 6 \\[35pt]
2. & 
\begin{minipage}{2.3cm}
\centering
\scalebox{0.4}{
\begin{tikzpicture}[mainnode/.style={circle, draw, fill=black, inner sep=1.6pt},whiskernode/.style={circle, draw=blue, fill=blue, inner sep=1.6pt},whisker/.style={blue, line width=0.5pt}]
\node[mainnode] (c) at (0,0) {};
\node[mainnode] (n1) at (0,1) {};
\node[mainnode] (n2) at (3,0) {};
\node[mainnode] (s1) at (0,-1) {};
\node[mainnode] (s2) at (2,1) {};
\node[mainnode] (e1) at (1,0) {};
\node[mainnode] (e2) at (2,0) {};
\node[mainnode] (w1) at (-1,0) {};
\node[mainnode] (w2) at (-2,0) {};
\draw (c) -- (n1);
\draw (c) -- (s1);
\draw (c) -- (e1) -- (e2) -- (n2);
\draw (c) -- (w1) -- (w2);
\draw (e2)-- (s2);
\end{tikzpicture}
}
\end{minipage} &
\begin{minipage}{2cm}
\centering
\scalebox{0.4}{
\begin{tikzpicture}[mainnode/.style={circle, draw, fill=black, inner sep=1.6pt},whiskernode/.style={circle, draw=blue, fill=blue, inner sep=1.6pt},whisker/.style={blue, line width=0.5pt}]
\node[mainnode] (c) at (0,0) {};
\node[mainnode] (n1) at (0,1) {};
\node[mainnode] (n2) at (3,0) {};
\node[mainnode] (s1) at (0,-1) {};
\node[mainnode] (s2) at (2,1) {};
\node[mainnode] (e1) at (1,0) {};
\node[mainnode] (e2) at (2,0) {};
\node[mainnode] (w1) at (-1,0) {};
\node[mainnode] (w2) at (-2,0) {};
\draw (c) -- (n1);
\draw (c) -- (s1);
\draw (c) -- (e1) -- (e2) -- (n2);
\draw (c) -- (w1) -- (w2);
\draw (e2)-- (s2);
\draw[whisker] (c) -- (-0.5,-0.5);
\draw[whisker] (c) -- (-0.5,0.5);
\node[whiskernode] at (-0.5,-0.5) {};
\node[whiskernode] at (-0.5,0.5) {};
\draw[whisker] (n1) -- (0.5,1.5);
\draw[whisker] (n1) -- (-0.5,1.5);
\node[whiskernode] at (0.5,1.5) {};
\node[whiskernode] at (-0.5,1.5) {};
\draw[whisker] (n2) -- (3.5,0.5);
\draw[whisker] (n2) -- (3.5,-0.5);
\node[whiskernode] at (3.5,0.5) {};
\node[whiskernode] at (3.5,-0.5) {};
\draw[whisker] (s1) -- (0.5,-1.5);
\draw[whisker] (s1) -- (-0.5,-1.5);
\node[whiskernode] at (0.5,-1.5) {};
\node[whiskernode] at (-0.5,-1.5) {};
\draw[whisker] (s2) -- (2.5,1.5);
\draw[whisker] (s2) -- (1.5,1.5);
\node[whiskernode] at (2.5,1.5) {};
\node[whiskernode] at (1.5,1.5) {};
\draw[whisker] (e1) -- (1.5,0.5);
\draw[whisker] (e1) -- (1.5,-0.5);
\node[whiskernode] at (1.5,0.5) {};
\node[whiskernode] at (1.5,-0.5) {};
\draw[whisker] (e2) -- (2.5,0.5);
\draw[whisker] (e2) -- (2.5,-0.5);
\node[whiskernode] at (2.5,0.5) {};
\node[whiskernode] at (2.5,-0.5) {};
\draw[whisker] (w1) -- (-1.5,0.5);
\draw[whisker] (w1) -- (-1.5,-0.5);
\node[whiskernode] at (-1.5,0.5) {};
\node[whiskernode] at (-1.5,-0.5) {};
\draw[whisker] (w2) -- (-2.5,0.5);
\draw[whisker] (w2) -- (-2.5,-0.5);
\node[whiskernode] at (-2.5,0.5) {};
\node[whiskernode] at (-2.5,-0.5) {};
\end{tikzpicture}
}
\end{minipage} & 
5 & 5 & 6 & 6 & 7 \\
\bottomrule
\end{tabular}
\caption{Comparison of bounds for Castelnuovo-Mumford regularity.}
\label{tab:reg-ub-comparison}
\end{table}
}
\end{Example}
\section{Conclusion and Future Directions}\label{Conclusion}

This work establishes new bounds for the induced matching number $\operatorname{im}(T)$ and the independence number $\alpha(T)$ of trees, addressing a gap in the literature. Our results reveal that Castelnuovo-Mumford regularity of edge ideals for trees and their multi-whiskered variants is  bounded by elementary combinatorial invariants. We provide the systematic bounds linking $\operatorname{im}(T)$ and $\alpha(T)$ to the order, diameter and number of pendant vertices.

These findings give rise to several compelling research directions that merit further exploration. A primary objective is the complete characterization of tree families that achieve equality in our bounds, which would provide deeper insight into the extremal behavior of Castelnuovo-Mumford regularity. This investigation would seek to identify the specific structural properties that force Castelnuovo-Mumford regularity to its minimum and maximum values relative to the fundamental parameters, namely order, diameter, and pendant vertex count.

Beyond these specific contributions, our work demonstrates how homological invariants can be bounded using elementary graph invariants. Looking forward, we anticipate that the perspectives established in this work, particularly the new bounds on independence numbers, will inspire further research at the interface of graph theory and homological algebra, continuing to reveal connections between discrete structures and algebraic behavior.


\begin{thebibliography}{99}
\bibitem{dao} Dao, H., Huneke, C.,  Schweig, J. (2013). Bounds on the regularity and projective dimension of ideals associated to graphs. Journal of Algebraic Combinatorics, 38(1), 37-55.
\bibitem{JCT1} Francisco, C. A., Hà, H. T. (2008). Whiskers and sequentially Cohen–Macaulay graphs. Journal of Combinatorial Theory, Series A, 115(2), 304-316.
\bibitem{JCT2} Herzog, J., Hibi, T., Zheng, X. (2006). Cohen–Macaulay chordal graphs. Journal of Combinatorial Theory, Series A, 113(5), 911-916.
\bibitem{hanvan} Hà, H. T.,  Van Tuyl, A. (2008). Monomial ideals, edge ideals of hypergraphs, and their graded Betti numbers. Journal of Algebraic Combinatorics, 27(2), 215-245.	
\bibitem{hibikrull} Hibi, T., Kanno, H.,  Matsuda, K. (2019). Induced matching numbers of finite graphs and edge ideals. Journal of Algebra, 532, 311-322.
\bibitem{HOA} Hoa, L. T.,  Tam, N. D. (2010). On some invariants of a mixed product of ideals. Archiv der Mathematik, 94(4), 327-337.
\bibitem{Mujahid} Hoang, D. T.,  Pham, M. H., Trung, T. N. (2024). The regularity and unimodality of h-polynomial of corona graphs. Journal of algebra and its applications, 2550343.
\bibitem{kat} Katzman, M. (2006). Characteristic-independence of Betti numbers of graph ideals. Journal of Combinatorial Theory, Series A, 113(3), 435-454.
\bibitem{kumini2009}  Kummini, M. (2009). Regularity, depth and arithmetic rank of bipartite edge ideals. Journal of Algebraic Combinatorics, 30(4), 429–445.
\bibitem {Mahmoudi} Mahmoudi, M., Mousivand, A., Crupi, M., Rinaldo, G., Terai, N., Yassemi, S. (2011). Vertex decomposability and regularity of very well-covered graphs. Journal of Pure and Applied Algebra, 215(12), 2473–2480.
\bibitem{moreyvir} Morey, S.,  Villarreal, R. H. (2012). Edge ideals: algebraic and combinatorial properties. Progress in commutative algebra, 1, 85-126.
\bibitem{Muta} Muta, Y., Pournaki, M. R., Terai, N. (2024). A local cohomological viewpoint on edge rings associated with multi-whisker graphs. Communications in Algebra, 53(5), 1856–1865.
\bibitem{Ahtsham} Shaukat, B., Haq, A. U., Ishaq, M. (2022). Some algebraic invariants of the residue class rings of the edge ideals of perfect semiregular trees. Communications in Algebra, 51(12), 1-20.
\bibitem{Moradi} Moradi, S., Kiani, D. (2012). Bounds for the regularity of edge ideal of vertex decomposable and shellable graphs. Bulletin of the Iranian Mathematical Society, 36(2), 267-277.
\bibitem{fakhari} Seyed Fakhari, S. A. (2025). On the regularity of squarefree part of symbolic powers of edge ideals. Journal of Algebra, 665, 103-130.
\bibitem{circulent} Uribe-Paczka, M. E.,  Van Tuyl, A. (2019). The regularity of some families of circulant graphs. Mathematics, 7(7), 657.
\bibitem{VTSequentially} Van Tuyl, A. (2009). Sequentially Cohen-Macaulay bipartite graphs: Vertex decomposability and regularity. Archiv der Mathematik, 93, 451-459.
\bibitem{vil1990} Villarreal, R. H. (1990). Cohen-Macaulay graphs. Manuscripta Mathematica, 66, 277-293.
\bibitem{woodroof} Woodroofe, R. (2014). Matchings, coverings, and Castelnuovo-Mumford regularity. Journal of Commutative Algebra, 6(2), 287-304.

			
		
		\end{thebibliography}
	\end{document}